\documentclass{amsart}

\usepackage[utf8]{inputenc}
\usepackage[T1]{fontenc}

\usepackage{amsmath,amsthm,amssymb,bbm}
\usepackage[foot]{amsaddr}
\usepackage{hyperref}
\usepackage{todonotes}
\usepackage{enumerate}

\theoremstyle{plain}
\newtheorem{thm}{Theorem}

\newtheorem{lem}[thm]{Lemma}
\newtheorem{cor}[thm]{Corollary}

\theoremstyle{definition}

\newtheorem{rem}[thm]{Remark}

\newcommand{\N}{\mathbb{N}}
\newcommand{\R}{\mathbb{R}}
\newcommand{\Z}{\mathbb{Z}}
\renewcommand{\S}{\mathcal{S}}
\newcommand{\zerovec}{o}
\newcommand{\zeromat}{\mathbb{O}}

\newcommand{\id}{\mathbbm{I}}
\newcommand{\conv}{\operatorname{conv}}

\newcommand{\xc}{\operatorname{xc}}
\newcommand{\sxc}{\operatorname{sxc}}

\newcommand{\dist}{\operatorname{dist}}
\newcommand{\prob}{\operatorname{Prob}}

\newcommand{\cP}{\mathcal{P}}

\newcommand{\setcond}[2]{\left\{ #1 \,:\, #2 \right\}}

\renewcommand{\mid}{\,:\,}
\newcommand{\ball}{\mathbb{B}}
\newcommand{\eps}{\varepsilon}

\newcommand{\transpose}[1]{{#1}^{\top}}

\newcommand{\cV}{\mathcal{V}}
\newcommand{\cB}{\mathcal{B}}
\newcommand{\floor}[1]{\left\lfloor #1 \right\rfloor}

\renewcommand{\phi}{\varphi}

\newenvironment{smallpmatrix}{\left(\begin{smallmatrix}}{\end{smallmatrix}\right)}

\title[Maximum Extension Complexity]{Maximum Semidefinite and Linear Extension Complexity of Families of Polytopes}
\author{Gennadiy Averkov$^1$}
\author{Volker Kaibel$^1$}
\author{Stefan Weltge$^2$}
\address{$^1$ Otto-von-Guericke-Universität Magdeburg, Germany}
\address{$^2$ ETH Zürich, Switzerland}
\email{averkov@ovgu.de}
\email{kaibel@ovgu.de}
\email{stefan.weltge@ifor.math.ethz.ch}

\begin{document}
    \begin{abstract}
        We relate the maximum semidefinite and linear extension complexity of a family of polytopes to the cardinality
        of this family and the minimum pairwise Hausdorff distance of its members.
        This result directly implies a known lower bound on the maximum semidefinite extension complexity of
        0/1-polytopes.
        We further show how our result can be used to improve on the corresponding bounds known for polygons with
        integer vertices.

        Our geometric proof builds upon nothing else than a simple well-known property of maximum volume inscribed
        ellipsoids of convex bodies.
        In particular, it does not rely on factorizations over the semidefinite cone and thus avoids involved procedures
        of balancing them as required, e.g., in~\cite{BrietDP15}.
        We hope that revealing the geometry behind the phenomenon opens doors for further results.

        Moreover, we show that the linear extension complexity of every $ d $-dimen\-sional 0/1-polytope is bounded from
        above by~$ O(\frac{2^d}{d}) $.
    \end{abstract}
    \maketitle
    \section{Introduction}

%
In what follows, let $d, k, \ell, m, n \in \N$.
Consider the vector space $\S^k$  of $k \times k$ symmetric real matrices and the convex cone $\S_+^k$  of
positive semidefinite matrices in $\S^k$. We consider representations
\begin{align}
\label{eqSemRep}
&P = \phi(Q), \ \text{where} 
\ Q = \setcond{x \in \R^n}{M(x) \in \S_+^k}, \ \text{and}
\\ \nonumber & \text{$\phi : \R^n \to \R^d$ and $M : \R^n \to \S^k$ are affine maps,} 
\end{align}
of sets $P \subseteq \R^d$. Note that $M(x)$ is a $k \times k$ symmetric matrix whose components are affine functions in $x$. The condition $M(x) \in \S_+^k$ is called a \emph{linear matrix inequality (LMI) of size $k$}, the set $Q$ defined by this condition is called a \emph{spectrahedron} and the affine image $P$ of $Q$ is called a \emph{projected spetrahedron}. See also \cite{HeltonNie2010,HeltonNie2012,Scheiderer2011} for a discussion of properties of spectrahedra and projected spectrahedra, and an example in Fig.~\ref{figEx}. We call \eqref{eqSemRep} an \emph{extended formulation of $P$ with an LMI of size $k$}.
If, additionally, $M(x)$ has the block-diagonal structure
\begin{align}
M(x) & = \begin{smallpmatrix} 
M_1(x) & & 
\\	& \ddots &
\\	&& M_\ell(x)
\end{smallpmatrix},
 &
\text{where $M_1,\ldots,M_\ell : \R^n \to \S^m$ and $\ell m = k$,} \label{eqBlocks}
\end{align}
then the LMI $M(x) \in \S_+^k$ can be reformulated as a system $M_1(x),\ldots,M_\ell(x) \in \S_+^m$ of $\ell$ LMIs. We call \eqref{eqSemRep}--\eqref{eqBlocks} an \emph{extended formulation of $P$ with $\ell$ LMIs of size $m$.} If $m=1$, the constraints $M_1(x),\ldots, M_\ell \in \S_+^m $ are merely linear inequalities and so the sets $Q$ and $P$ are polyhedra. We call  the representation \eqref{eqSemRep}--\eqref{eqBlocks} with $m=1$ an \emph{extended formulation of the polyhedron $P$ with $\ell$ linear inequalities}.

\setlength{\unitlength}{1mm}
\begin{figure}
	\label{figEx}
	\begin{picture}(80,55)
	\put(-30,0){\includegraphics[width=80mm]{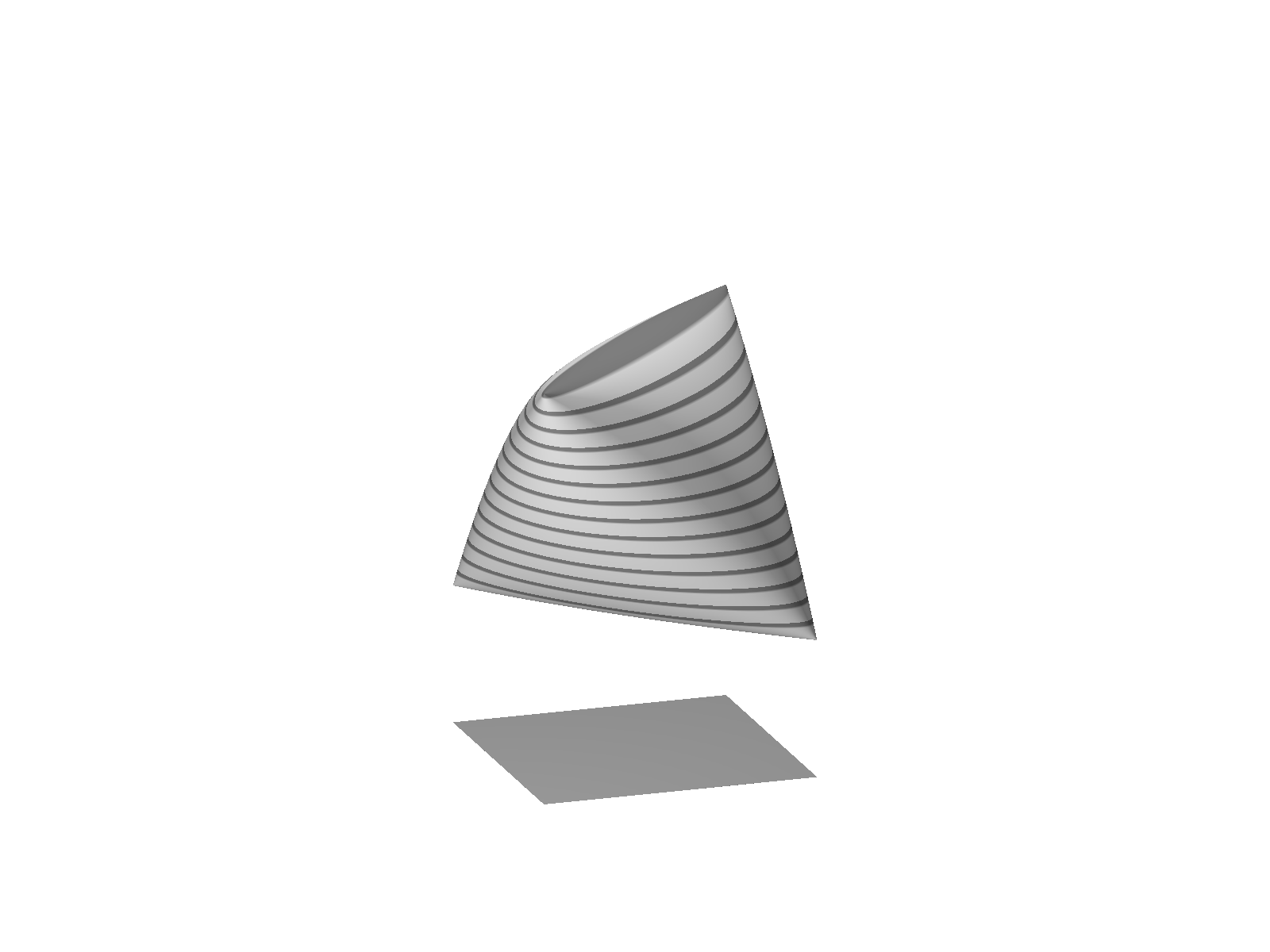}}
	\put(25,30){ $Q = \setcond{x \in \R^3}{\begin{smallpmatrix} 1 & x_1 & x_2 \\ x_1 & 1 & x_3 \\ x_2 & x_3 & 1 \end{smallpmatrix} \in \S_+^3}$}
	\end{picture}
	\caption{The orthogonal projection of the depicted spectrahedron $Q$ onto a horizontal plane is a square. Thus, a square has an extended formulation with an LMI of size $3$.}
\end{figure}

As the problem of optimizing a linear function over $ P $ can be converted into optimizing a linear function over $ Q$,
the extended formulation \eqref{eqSemRep} may be of advantage if $ Q$ has a simpler description than the original description of $ P $.
Thus, one is interested in finding small linear and semidefinite formulations of $P$. 

Since polytopes are of particular importance in discrete optimization and since our motivation originated from this area, we concentrate on the case that $P$ is a polytope. In this case, we call the smallest $k$ such that $P$ has an extended formulation with an LMI of size $k$ the \emph{semidefinite extension complexity} of $P$ and denote this value by $\sxc(P)$. 
Similarly, the smallest $\ell$ such that $P$ has an extended formulation with $\ell$ linear inequalities is called the \emph{linear extension complexity} of $P$ and is
denoted by $\xc(P)$. If $P$ is empty or a single point, we let $\sxc(P)=\xc(P)=0$. Note that $\sxc(P) \le \xc(P)$. 
For more information and examples, we refer to the surveys of Kaibel~\cite{Kaibel11}, Conforti, Cornuéjols \&
Zambelli~\cite{ConfortiCZ13}, Gouveia, Parrilo \& Thomas~\cite{GouveiaPT13} and Fawzi, Gouveia, Parrilo, Robinson \& Thomas \cite{FGPRT2015}.

Semidefinite and linear extension complexities of various \emph{specific} polytopes arising in
optimization have been extensively studied; see, e.g.,
\cite{Goemans14,FioriniMPTdW12,KaibelW15,PokuttaV13,AvisT15,Rothvoss14,LeeRS15}.
For understanding the power of extended formulations in general, it is also interesting to provide bounds on
extension complexities for \emph{families} of polytopes.
First results of this type were obtained by Rothvoß~\cite{Rothvoss13} and Fiorini, Rothvoß \&
Tiwary~\cite{FioriniRT12} who established lower bounds on the maximum linear extension complexity of $ 0/1 $-polytopes
and convex $ n $-gons, respectively.
Later, their results were carried over to the semidefinite
case by Bri{\"e}t, Dadush \& Pokutta~\cite{BrietDP15}.
As all these bounds are obtained by counting arguments, they are remarkable in the sense that no specific polytopes attaining the respective bounds are known so far.

While the approaches in~\cite{Rothvoss13,FioriniRT12,BrietDP15} can be applied to further families of polytopes, it
seems that, dealing with a new family, one is forced to repeat large parts of the argumentation in the above sources. In contrast, in this paper, we present a theorem which can be used as a simple tool for finding lower bounds on the maximum semidefinite and  linear extension
complexity for general families of polytopes.

\newcommand{\INF}{\operatorname{\vphantom{sup}\inf}\limits}

In $\R^d$, we consider the standard Euclidean norm $ \|\,.\,\| $ and the $d$-dimensional unit ball $ \ball^d := \{ x \in \R^d : \|x\| \le 1 \} $.
Given non-empty compact sets $ X,Y \subseteq \R^d $, their \emph{Hausdorff distance} with respect to the Euclidean norm is defined by
\[
    \dist(X,Y) := \max \Bigl\{ \, \sup_{x \in X} \INF\limits_{y \in Y} \|x-y\| \, , \,
    \sup_{y \in Y} \INF\limits_{x \in X} \|x-y\| \, \Bigr\}.
\]
We use $\log$ to denote the logarithm to the base $2$. 

\begin{thm}
	\label{thmMain}
	Let $\cP$ be a family of polytopes in $\R^d$ of dimensions at least one with $2 \le |\cP| < \infty$ such that each $P
	\in \cP$ has an extended formulation with $\ell$ LMIs of size $m$. Let $\rho>0$ and $\Delta>0$
	be such that each $P \in \cP$ is contained in the ball $\rho \ball^d$ and, for every two distinct polytopes $P \in \cP$ and
	$P' \in \cP$, one has $\dist(P,P') \ge \Delta$.
	Then 
	\begin{equation}
		\label{eqMainBound}
		\ell^2 m^4 \ge \frac{\log |\cP| }{8 d \left(1 + \log (2\rho/\Delta) + \log \log |\cP| \right) } =: B.
	\end{equation}
	In particular, we have
	\begin{align*}
		& \max_{P \in \cP} \sxc(P) \ge \sqrt[4]{B} & &\text{and} & \max_{P \in \cP} \xc(P) \ge \sqrt{B}.
	\end{align*}
\end{thm}

For deriving lower bounds on extension complexities for a concrete family $\cP$, it
suffices to choose $\rho$ and $\Delta$ appropriately, to bound $\log |\cP|$ in the enumerator from below and to bound $\log \log |\cP|$ in the denominator from above. 
In Section~\ref{secApplications} we demonstrate how the mentioned results in~\cite{Rothvoss13,FioriniRT12,BrietDP15} can
be easily derived in this way. 

Besides the simple applicability of Theorem~\ref{thmMain}, we view its short and simple proof as an essential
contribution.
Note that the original proofs in~\cite{Rothvoss13,FioriniRT12,BrietDP15} turn out to be quite long and require a number of
non-trivial tools.
They rely on a counting argument developed in~\cite{Rothvoss13} based on encoding extended formulations by certain kinds
of factorizations of slack matrices of polytopes (see~\cite{Yannakakis91,Kaibel11,GouveiaPT13}) whose components have
to be carefully balanced and rounded. This requires several technical steps.
In contrast, our geometric proof of Theorem~\ref{thmMain} builds upon nothing else than a well-known property of
maximum-volume inscribed ellipsoids of convex bodies and simple linear algebra.

We give a short overview of the results in \cite{Rothvoss13}, \cite{FioriniRT12} and \cite{BrietDP15}, which we
reprove in our paper.
Rothvoß~\cite{Rothvoss13} proved that the maximum linear extension complexity of a $ 0/1 $-polytope in $ \R^d $ is
exponential in $ d $.
Recall that a \emph{$0/1$-polytope} in $\R^d$ is the convex hull of a subset of $\{0,1\}^d$.
Bri{\"e}t, Dadush \& Pokutta~\cite{BrietDP15} improved on this result by showing that even the maximum semidefinite
extension complexity of $ 0/1 $-polytopes in $ \R^d $ is exponential in $ d $.
The authors in~\cite{BrietDP15} mention that their arguments actually imply that the vast majority of $ 0/1 $-polytopes
in $ \R^d $ have semidefinite extension complexities that are exponential in $ d $.
In Corollary~\ref{corRandomZeroOnePolytopes} we give an explicit formulation and proof of this fact.

While $ 0/1 $-polytopes are ubiquituous in optimization, the interest in the family of $ n $-gons (i.e., two-dimensional
polytopes with $ n $ vertices) stems from the fact that, despite their trivial facial structure, the exact
asymptotics of the maximum extension complexity of $ n $-gons is not known (both in the linear and the semidefinite case).
It is known that the maximum linear extension complexity of $n$-gons is sublinear in $n$; see Shitov~\cite{Shitov14}. On the other hand,  
Fiorini, Rothvoß \& Tiwary~\cite{FioriniRT12} provided a lower bound of order $\Omega(\sqrt{n})$.
Recently, it was shown that this bound is also achieved for $n$-gons with vertices in $\Z^2$; see Padrol~\cite{Padrol2016}.
Bri{\"e}t, Dadush \& Pokutta~\cite{BrietDP15} showed that the maximum semidefinite extension complexity of
$n$-gons with vertices in $\Z^2$ is of order $\Omega(\sqrt[4]{n/\log n})$.
Using Theorem~\ref{thmMain} we give a simple proof of the fact that among polygons with vertices in $ \Z^2
$ there exist $n$-gons with linear extension complexity of order $ \Omega(\sqrt{n}) $ as well as $n$-gons with semidefinite extension complexity
of order $ \Omega(\sqrt[4]{n}) $, thus slightly improving on the previously known bounds; see Corollary~\ref{corNGons}.

Finally, we conclude our paper by giving an upper bound on linear extension complexities of $ 0/1 $-polytopes.
It is known that the extension complexity of a polytope is bounded by the number of its vertices.
Thus, the linear extension complexity of a $ 0/1 $-polytope in $ \R^d $ is of order $ O(2^d) $.
Surprisingly, no other bound than this trivial one seems to have been available so far.
For this reason, we show that the linear extension complexity of every $0/1$ polytope in $\R^d$ is at most $\frac{9}{d} 
2^d$ if $ d \ge 4$; see Section~\ref{secUpperBound}.

\subsubsection*{Notation}
Let $\N:=\{1,2,3,\ldots\}$. Throughout the paper, $d,k,\ell,m,n \in \N$. We define $ [n] := \{1,\dotsc,n\} $. 
The identity matrix of size $k \times k$ is denoted by $\id_k$. If the size of the identity matrix is clear from the context, we omit the subscript and write $\id$. Zero vectors are denoted by $\zerovec$ , while zero matrices are denoted by $\zeromat$. Their sizes will be clear from the context.

    \section{Normalization of extended formulations}

We call the extended formulation \eqref{eqSemRep} \emph{normalized} if $M(\zerovec)=\id$ and $\ball^n \subseteq Q \subseteq n \ball^n$. In this section, we show that extended formulations of a polytope can be converted into a normalized form.

\begin{lem}
    \label{lemBounded}
    Let $ C \subseteq \R^n $ be a closed convex set and $ \phi : \R^n \rightarrow \R^d $ be an affine map such that $
    P := \phi(C) $ is a polytope.
    Then there is an affine subspace $ L $  of $ \R^n $ such that $ C \cap L $ is bounded with $ P = \phi(C \cap L)
    $.
\end{lem}
\begin{proof}
    We argue by induction on $ n $.
    For $ n = 1 $, the assertion is easy to verify.
    Let $ n \ge 2 $ and assume that the assertion has been verified for closed convex subsets of $\R^{n-1}$.
    If $ C \subseteq \R^n $ is bounded, the assertion is trivially fulfilled with $ L = \R^n $.
    Consider the case of unbounded $ C $.
    In this case, there exists a non-zero vector $ u \in \R^n $ such that $ x + \mu u \in C $ for every $ x \in C $ and
    every $ \mu \ge 0 $; see~\cite[Thm.~8.4]{Rockafellar70}.
    Let $ x_1,\dotsc,x_\ell $ be the vertices of $ P $ and fix points $ y_1,\ldots,y_\ell \in C $ with $ \phi(y_i) =
    x_i $ for $i \in [\ell]$.
    By the choice of $u$, for each $i \in [\ell]$, the ray $R_i:=\setcond{y_i+ \mu u}{\mu \ge 0}$ in direction $u$
    emanating from $y_i$ is a subset of $C$.
    The image $\phi(R_i)$ of the ray $ R_i $ is either a ray or a point, and since $P$ is bounded, we have $
    \phi(R_i) = \{x_i\} $.
    Choose a hyperplane $ H $ in $ \R^n $ orthogonal to $ u $ that meets all the finitely many rays $ R_1,\ldots,R_\ell
    $.
    By construction, $ \phi(C \cap H) = P $.
    The set $ C \cap H $ is a closed convex subset of $H$.
    Since $H$ can be identified with $\R^{n-1}$, the induction assumption yields the existence of an affine
    subspace $ L $ of $ H $ such that $ C \cap L $ is bounded and $ \phi(C \cap L) = P $.
\end{proof}
The following lemma follows from a basic result from the theory of convex sets.
\begin{lem}
    \label{lemSandwich}
    Let $ Q $ be a compact convex subset of $ \R^n $ with non-empty interior.
    Then there exists an affine bijection $ \phi : \R^n \rightarrow \R^n $ such that $ \ball^n \subseteq \phi(Q)
    \subseteq n \ball^n $.
\end{lem}
\begin{proof}
    Consider the so-called John-L\"owner ellipsoid $E$ of $Q$, that is, $E$ is the ellipsoid of maximum volume contained
    in $Q$.
    Let $c$ be the center of $E$.
    It is well-known that $E-c \subseteq Q - c \subseteq n (E - c)$; see, for example, \cite[Chap.~V
    Thm.~2.4]{Barvinok02}.
    Thus, one can choose $\phi$ to be an affine bijection with $\phi(E) = \ball^n$.
\end{proof}

\begin{lem}[{Helton \& Vinnikov \cite[Lem.~2.3]{HeltonVinnikov2007}}]
	Let $Q=\setcond{x \in \R^n}{M(x) \in \S_+^k}$ be a spectrahedron given by an LMI $M(x) \in \S_+^k$. If $\zerovec$ is in the interior of $Q$, then $Q$ can also be written as $Q= \setcond{x \in \R^n}{A(x) + \id \in \S_+^k}$, where $A : \R^n \to \S^k$ is a linear map. 
\end{lem}
\begin{proof}
	We present the argument from \cite{HeltonVinnikov2007} for the sake of completeness.  Let $M(x)=S(x)+ T$, where $S : \R^n \to \S^k$ is a linear map and $T \in \S^k$. We first show that, for each $x \in \R^n$, the kernel of $T$ is a subspace of the kernel of $S(x)$. Fix an arbitrary $x$ and an arbitrary $u$ in the kernel of $T$. As the origin of $\R^n$ is in the interior of $Q$ the matrices $S(\pm \eps  x) + T = \pm \eps S(x) + T$ are positive semidefinite, for a sufficiently small $\eps>0$. In particular, both values $\transpose{u}(\pm \eps  S(x) u +  T) u$ are non-negative. Since $Tu =\zerovec$, we arrive at $\transpose{u} S(x) u = 0$. We have shown that for the positive-semidefinite matrix $\eps S(x)+ T$, one has $\transpose{u} (\eps S(x) + T) u = 0$. This means that $u$ is in the kernel of $\eps S(x)+ T$. Since $u$ is in the kernel of $T$, we conclude that $u$ is in the kernel of $S(x)$. 
	
	Since $\zerovec$ is in $Q$ the matrix $T$ is positive semidefinite. Thus, there exists an invertible matrix $U$ such that 
	\(
		\transpose{U} T U = \begin{smallpmatrix} \id_{r} & \zeromat & \\ \zeromat & \zeromat \end{smallpmatrix},
	\)
	where $r$ is the rank of $T$. The last $m-r$ columns of $U$ belong to the kernel of $T$ and so we have 
	\(
		\transpose{U} S(x) U = \begin{smallpmatrix} S'(x) & \zeromat & \\ \zeromat & \zeromat \end{smallpmatrix}
	\)
	for some linear map $S' : \R^n \to \S^r$. This shows that the condition $S(x) + T \in \S^m$ is equivalent to $S'(x) + \id \in \S^r$. The latter condition is equivalent to $A(x) + \id \in \S^k$ with a linear map $A : \R^n \to \S^k$ given by 
	\(
		A (x) = \begin{smallpmatrix} S'(x) & \zeromat & \\ \zeromat & \id \end{smallpmatrix}.
	\)
\end{proof}

\begin{thm}	\label{thmNormalized}
	Let $P \subseteq \R^d$ be a polytope with $\dim(P) \ge 1$. If $P$ has an extended formulation with $\ell$ LMIs of size $m$, then $P$ also has a normalized extended formulation with $\ell$ LMIs of size $m$. 
\end{thm}
\begin{proof}
	Consider an arbitrary extended formulation \eqref{eqSemRep}. 	By Lemma~\ref{lemBounded} there exists an affine subspace $L$ of $\R^n$ such that the set $Q \cap L$ is bounded and satisfies $P =\phi(Q \cap L)$. Let $n' :=\dim(Q \cap L)$ and choose a set
	$Q' \subseteq \R^{n'}$ affinely isomorphic to $Q \cap L$. That is, for some affine map $\psi : \R^{n'} \to \R^n$ the set $Q'$ is bijectively mapped onto $Q \cap L$ by $\psi$. By Lemma~\ref{lemSandwich}, without loss of generality, we can choose $Q'$ appropriately so that the inclusions $\ball^{n'} \subseteq Q' \subseteq n' \ball^{n'}$ are fulfilled. We have $P=\phi(Q \cap L) = \phi(\psi(Q'))$, where $Q'$ is a spectrahedron given by 
	\[
		Q' = \setcond{x' \in \R^{n'}}{\psi(x') \in Q} = \setcond{x' \in \R^{n'}}{M(\psi(x')) \in  \S_+^k}.
	\]
	Now, assume that $M(x)$ has the block-diagonal structure \eqref{eqBlocks}. Then we can write $Q'$ as $Q'= Q'_1 \cap \dotsb \cap Q'_\ell$, where $Q'_i := \setcond{x' \in \R^{n'}}{M_i(\psi(x')) \in \S_+^k}$ for $i \in [\ell]$. 
	
	Since $\ball^{n'} \subseteq Q'$, the origin of $\R^{n'}$ is in the interior of each $Q'_i$. Application of Lemma~\ref{lemBounded} to $Q'_i$ yields the existence of a linear map $A_i : \R^{n'} \to \S_+^m$ with  $Q'_i = \setcond{x' \in \R^{n'}}{A_i(x') + \id \in \S_+^k}$. We have thus constructed a normalized extended formulation $P = \setcond{\phi(\psi(x'))}{x' \in \R^{n'}, A_1(x') + \id,\ldots,A_\ell(x')+ \id \in \S_+^k}$ of $P$ with $\ell$ LMIs of size $m$.
\end{proof}

    \section{Proof of the main theorem} 

\label{secMainProof}

We have already fixed the norm on $\R^d$ to be the Euclidean norm. We also use the following norms for a matrix $T \in \S^m$, a linear map $\phi : \R^n \to \R^d$, and a linear map $A : \R^n \to \S^m$: 
\begin{align*}
\|T\| & := \max_{x \in \ball^n} \|T x\|, 
& \|\phi\| & := \max_{x \in \ball^n} \|\phi(x)\|,
& \|A\| & := \max_{x \in \ball^n} \|A(x)\| .
\end{align*}
All three norms are the so-called operator norms.  It is well-known that $\|T\|$ is equal to the largest absolute value of an eigenvalue of $T$. This also shows that $\| A(x) \|$ is the largest absolute value of an eigenvalue of $A(x)$ among all $x \in \ball^n$. 

In what follows, for encoding normalized extended formulations we use the vector space $\cV^{\ell,m,n}_d$ of triples $(A,\phi,t)$ such that $A : \R^n \to \S^k$ and $\phi : \R^n \to \R^d$ are linear maps, $t \in \R^d$ and $ A $ consists of $ \ell $ blocks of size $ m \times m $ (thus $ k = \ell m $). Each such triple $(A,\phi,t)$ determines a subset $P=\setcond{\phi(x)+t}{x \in \R^n, A(x) + \id \in \S_+^k}$. Every normalized extended formulation with $\ell$ LMIs of size $m$ can be encoded using an appropriate choice of $(A,\phi,t) \in \cV^{\ell,m,n}_d$.

\begin{lem}
	\label{lemBounds}
	Let $P$ be a compact subset of $\R^d$ that has a normalized extended formulation 
	\[
	P= \setcond{\phi(x)+t}{x \in \R^n, \ A(x) +\id \in \S_+^k},
	\]
	with $(A,\phi,t) \in \cV^{\ell,m,n}_d$. Let $\rho>0$ be such that $P$ is contained in the ball $\rho \ball^d$. Then 
	\begin{align}
	\label{eqBounds}
	\|A\| & \le 1, & \|\phi \| & \le \rho,  & \|t\| & \le \rho, & &\text{and} & n &\le \ell m^2.
	\end{align}

	Furthermore, if $P'$ is another compact subset of $\rho \ball^n$ having a normalized extended formulation 
	\[
	P'= \setcond{\phi'(x)+t' \in \R^n}{x \in \R^n, \ A'(x)+ \id \in \S_+^k}
	\] with $(A',\phi',t') \in \cV_d^{\ell,m,n}$, then
	\begin{equation}
	\label{eqDistanceBound}
	\dist(P,P') \le \rho n^2 \|A-A'\| + n \|\phi - \phi'\| + \|t-t'\| 
	\end{equation}
\end{lem}
\begin{proof}
	We define the spectrahedron $Q=\setcond{x\in \R^n}{A(x) + \id \in \S_+^k}$.
	For showing \eqref{eqBounds} consider an arbitrary $ x \in \ball^n $. Since $ \ball^n \subseteq Q $, we have $  A(\pm x) + \id \in \S_+^k $. Thus, for every eigenvalue $\lambda \in \R$ of $A(x)$ and a corresponding eigenvector $u$ of unit length, one gets $\pm \lambda + 1 =  \transpose{u} ( A(\pm x) +\id) u \ge 0$. Hence $|\lambda| \le 1$ and we have thus shown $\|A\| \le 1$. 
	Since $ \pm x \in \ball^n \subseteq Q$, one has $ \|\phi(x) + t\| \le \rho $ and $ \|\phi(-x) + t\|
	= \|t - \phi(x)\| \le \rho $. Thus, setting $ x = \zerovec $, one obtains $ \|t\| \le \rho $. We also have $ \|\phi(x)\| \le \tfrac{1}{2} \|t + \phi(x)\| + \tfrac{1}{2} \|t -
	\phi(x)\| \le \rho $, which yields $ \|\phi\| \le \rho $.
	
	In order to show $n \le \ell m^2$, assume that one had $n > \ell m^2$. Then $\dim(\R^n) > \dim(A(\R^n))$. Thus, $A$ maps a nonzero vector $x$ of $\R^n$ to a zero matrix. For such a vector $x$ one has $\alpha x \in Q$ for every $\alpha \in \R$. The latter contradicts the inclusion $Q \subseteq n \ball^n$ in the definition of the normalized extended formulation.
	
	It remains to show \eqref{eqDistanceBound}. We define the spectrahedron $Q' = \{ x\in \R^n : A'(x) + \id \in \S_+^k \} $ corresponding to $P'$. 
	Due to the symmetry between $P$ and $P'$ in the definition of the Hausdorff distance, it suffices to show that for every point $ y \in P $ there exists a point
	$ y' \in P' $ with
	\[
	\|y - y'\| \le \rho n^2 \|A-A'\| + n \|\phi - \phi'\| + \|t-t'\|.
	\]
	Let $y=\phi(x) + t$ with $ x \in Q $, and define $y'$ by $y' := \phi'(x')+t$, where  $ x' := \lambda x $ and $ \lambda := \frac{1}{1 + n \|A - A'\|} \in (0,1] $.  We show that $ x' \in Q' $ and so $y' \in P'$.
	We have to show that $ A'(x') + \id $ is positive semidefinite, i.e., that
	\begin{equation}
	\label{eqPSD1}
	\transpose{v} \big( A'(x') + \id \big) v \ge 0
	\end{equation}
	holds for every $ v \in \R^k $ with $ \|v\| = 1 $, which is equivalent to
	$
	\transpose{v} A'(x') v \ge -1.
	$
	Denoting $ D := A' - A $, for every $ v \in \R^k $ with $ \|v\| = 1 $ we indeed obtain
	\begin{align*}
	\transpose{v} A'(x') v
	& = \lambda \, \left( \transpose{v} D(x) v + \transpose{v} A(x) v \right) \\
	& \ge \lambda \, \left( \transpose{v} D(x) v - 1 \right) & & \text{(since $x \in Q$)}\\
	& \ge \lambda \, \left( - \|v\| \cdotp \|D(x) v\| - 1 \right) & & \text{(by Cauchy-Schwarz)}\\
	& \ge \lambda \, \left( - \|D(x)\| - 1 \right) \\
	& \ge \lambda \, \left( - \|D\| \cdotp \|x\| - 1 \right) \\
	& \ge \lambda \, \left( - n \|D\| - 1 \right) \\
	& = -1.
	\end{align*}
	Thus, $x' \in Q'$. We have 
	\begin{align*}
	\|y-y'\| & = \|(\phi(x) + t) - (\phi'(x') + t')\| \\
	& \le \|\phi(x) - \phi'(x')\| + \|t - t'\| \\
	& \le \|\phi - \phi'\| \cdotp \|x\| + \|\phi'\| \cdotp \|x - x'\| + \|t - t'\| \\
	& \le \|\phi - \phi'\| \cdotp n + \rho \|x-x'\| + \|t - t'\|, 
	\end{align*}
	where 
	\[
	\|x - x'\|
	= \left( 1 - \tfrac{1}{1 + n \|A - A'\|} \right) \|x\|
	\le \left( 1 - \tfrac{1}{1 + n \|A - A'\|} \right) n
	\le n^2 \|A - A'\|.
	\]
	This shows \eqref{eqDistanceBound}.
\end{proof}

\begin{proof}[Proof of Theorem~\ref{thmMain}]
	Let $N:=|\cP|$ and let $\cP=\{P_1,\ldots,P_N\}$.
	We consider an arbitrary $i \in [N]$. 
	Theorem~\ref{thmNormalized} implies that $P_i$ has a normalized extended formulation $P_i=\phi_i(Q_i) + t_i$, where $Q_i$ is the spectrahedron given by $Q_i= \{ x \in \R^{n_i} : A_i(x)+\id \in \S_+^m \} $ for some $n_i \in \N$ and $(A_i,\phi_i,t_i) \in \cV^{\ell,m,n_i}_d$. By Lemma~\ref{lemBounds}, $n_i \le \ell m^2$. Thus, for the sets $ W_n := \{ (A_i,\varphi_i,t_i) \mid i \in [N], \ n_i = n \} $ with $ n \in [\ell m^2] $, we have
	\begin{equation}
	\label{eqPartitionFamilyOfPolytopes}
	N = \sum_{n=1}^{\ell m^2} |W_n|
	\end{equation}
	We will now bound the cardinality of each $ W_n $.
	We fix $ n \in [\ell m ^2] $ and endow the vector space $\cV^{\ell,m,n}_d$ with the norm
	\[
	\|(A,\varphi,t)\| := \rho n^2 \|A \| + n \|\varphi \| + \|t\|.
	\]
	By inequality \eqref{eqDistanceBound} in Lemma~\ref{lemBounds},
	\(
	\|w - w'\| \ge \Delta
	\)
	holds for all $ w,w' \in W_n $ with $ w \ne w' $.
	This means that the open balls 
	\[
	\cB_w := \setcond{ v \in \cV^{\ell,m,n}_d }{ \|w - v\| < \Delta/2 }
	\]
	of the normed space $\cV^{\ell,m,n}_d$ with
	$ w \in W_n $ are pairwise disjoint.
	On the other hand, by inequalities  \eqref{eqBounds} in Lemma~\ref{lemBounds}, one has 
	\[
	\|w\| \le \rho n^2 \cdotp 1 + n \cdotp \rho + \rho \le 3 \rho n^2
	\]
	for every $ w \in W_n $.
	Thus, all balls $ B_w $ with $ w \in W_n $ are contained in the closed ball $ \cB := \{ v \in \cV^{\ell,m,n}_d : \|v\| \le 3 \rho n^2 + \Delta/2 \}
	$.
	Observe that for each $ w \in W_n $ the ratio of the volumes of $ \cB $ and $ \cB_w $ is $(\tfrac{3 \rho n^2 +
		\Delta/2}{\Delta/2})^{\dim(\cV^{\ell,m,n}_d)}$.
	The total volume of the disjoint balls $ \cB_w $ with $ w \in W_n $ is not larger than the volume of $ \cB $. The latter observation combined with the bound $\dim(\cV^{\ell,m,n}_d) = n \ell m(m+1)/2 + n d + d \le 3 d \ell^2 m^4 $ yields
	\[
	|W_n| \le \left( \tfrac{6 \rho n^2 + \Delta}{\Delta} \right)^{3 d \ell^2 m^4}.
	\]
	%
	The Hausdorff distance between two elements of $\cP$ is at most $2 \rho$, since every point of $\rho \ball^d$ is at distance at most $2 \rho$ to every other point of $\rho \ball^d$. Hence $\Delta \le 2 \rho$ and we obtain 
	\[
	|W_n| \le \left( \tfrac{6 \rho n^2 + 2 \rho}{\Delta} \right)^{3 d \ell^2 m^4} \le \left( \tfrac{8 \rho n^2}{\Delta} \right)^{3 d \ell^2 m^4} \le \left( \tfrac{8 \rho \ell^4 m^8}{\Delta} \right)^{3 d \ell^2 m^4}
	\]
	Using the notation $s= \ell m^2$, the latter bound can be written as $|W_n| \le \left( \tfrac{8 \rho s^4}{\Delta} \right)^{3 d s^2}.$

	In view of \eqref{eqPartitionFamilyOfPolytopes}, we get
	\[
	N \le s \left( \tfrac{8 \rho s^4}{\Delta} \right)^{3 d s^2} \le \left( \tfrac{8 \rho s^4}{\Delta} \right)^{4 d s^2}
	\]
	Taking the logarithm of the left and the right hand side, we arrive at $\log N \le 8 d s^2 ( 1 + \log (2\rho/\Delta) + \log (s^2))$. In the case $s^2> \log N $, \eqref{eqMainBound} is obviously fulfilled. In the case $s^2 \le \log N $, we use the estimate $\log(s^2) \le \log \log N $ and arrive at 
	$\log N \le 8 d s^2( 1 + \log (2 \rho / \Delta) + \log \log N)$, which shows that also in this case \eqref{eqMainBound} is fulfilled.
\end{proof}




\begin{rem}
	One can also consider more general extended formulations with $\ell$ semidefinite constraints of sizes $m_1,\ldots,m_\ell \in \N$, where $m_1,\ldots, m_\ell$ may not be equal. It is clear that Theorem~\ref{thmMain} can be generalized in a straightforward way to cover such more general formulations.
\end{rem}

\section{Applications}
\label{secApplications}

Given a finite set $X \subseteq \R^d$, we introduce the family $\cP(X) = \setcond{\conv(X')}{X' \subseteq X}$.
In particular, $\cP(\{0,1\}^d)$ is the set of all $ 0/1 $-polytopes in $ \R^d $.
\begin{cor}
    \label{corRandomZeroOnePolytopes}
    Let $ d \in \N $, $d \ge 3$, and let $ P $ be a random polytope uniformly distributed in $ \cP(\{0,1\}^d)$. Then one has
    \begin{align*}
	    \prob \left( \sxc(P) \le \frac{2^{d/4}}{3 \sqrt{d}}  \right)  & \le 2^{-2^{d-1}}
	    & &\text{and} & \prob \left( \xc(P) \le \frac{2^{d/2}}{9 d}  \right) & \le 2^{-2^{d-1}}.
    \end{align*}
\end{cor}
\begin{proof}
	We will apply Theorem~\ref{thmMain} for subfamilies of $\cP(\{0,1\}^d)$.  We can fix $\rho=\sqrt{d}$, since the maximum Euclidean norm of points from $\{0,1\}^d$ is $\sqrt{d}$. We can fix $\Delta=1/\sqrt{d}$, since $\dist(P_1,P_2) \ge
	\frac{1}{\sqrt{d}} $ holds for all non-empty polytopes $ P_1, P_2 \in \cP(\{0,1\}^d) $ with $ P_1 \ne P_2 $.
	To see this, consider a vertex of one of these two polytopes that does not belong to the other one.
	Without loss of generality, we assume that this vertex is the origin and that it belongs to $ P_1 $ but not to $ P_2 $.
	Then $\zerovec$ and  $ P_2 $ are separated by the hyperplane $ H := \{ (x_1,\ldots,x_d) \in \R^d \mid \sum_{i=1}^d x_i \ge 1
	\} $.
	The distance of $\zerovec$ to every point of $P_1$ is bounded from below by the distance of $\zerovec$ to $H$. Hence $\dist(P_1,P_2) \ge 1/\sqrt{d}$.
	
	For $ m,\ell \in \N $ let $\cP_{\ell,m}$ be a subfamily of $\cP(\{0,1\}^d)$ consisting of polytopes of dimension at least one which have an extended formulation with $\ell$ semidefinite constraints of size $m$. Since $|\cP_{\ell,m}| \le |\cP(\{0,1\}^d)| \le 2^{2^d}$, one has $\log \log |\cP| \le d$. Thus, Theorem~\ref{thmMain} yields 
	\[
		\ell^2 m^4 \ge \frac{\log |\cP_{\ell,m}|}{8d ( 1 + \log (2 d) + d)} \ge \frac{\log |\cP_{\ell,m}|}{32 d^2}.
	\]
	Hence $|\cP_{\ell,m}| \le 2^{32 \ell^2 m^4 d^2}$. Let $\cP'$ be the family of all $0/1$ polytopes $P'$ in $\R^d$ such that $P'$ is empty or a singleton. One has $|\cP'| =2^d + 1$. In view of  $|\cP_{\ell,m} \cup \cP'| \le 2^{33 \ell^2 m^4 d^2}$, we obtain
	\[
		\prob \left( P \in \cP_{\ell,m} \cup \cP'\right) \le 2^{33 \ell^2 m^4 d^2 - 2^d}.
	\]
	Hence
	\begin{align} \label{lmProbBound}
		& \prob \left( P \in \cP_{\ell,m} \cup \cP' \right) \le 2^{-2^{d-1}}  & & \text{if} & & 33 \ell^2 m^4 d^2 \le 2^{d-1}.
	\end{align}
	 The assertion for the semidefinite extension complexity is verified as follows. In the case  $\frac{2^{d/2}}{3 \sqrt{d}}<1$, we need to show that $\prob \left( \sxc(P)=0 \right)  = \frac{2^d+1}{2^{2^d}}$ is at most $2^{-2^{d-1}}$. The latter is true in view of $d \ge 3$. If $\frac{2^{d/2}}{3 \sqrt{d}} \ge 1$, the assertion follows from \eqref{lmProbBound} by setting $\ell=1$ and $m= \floor{\frac{2^{d/2}}{3 \sqrt{d}}}$. Analogously, to prove of the assertion for the linear extension complexity, we distinguish the two cases $\frac{2^{d/2}}{9d}<1$ and $\frac{2^{d/2}}{9 d} \ge 1$ and use \eqref{lmProbBound} with $\ell= \floor{\frac{2^{d/2}}{9d}}$ and $m=1$ in the second case. 
\end{proof}

\begin{cor}
    \label{corNGons}
    Let $ n \in \N $, $n \ge 2$, and let $\cP$ be the family of integral polygons $P \in \cP([n^2] \times [n^4])$ with $n$ vertices. Then one has 
    \begin{align*}
	    \max_{P \in \cP} \sxc(P) & \ge \frac{1}{4} \sqrt[4]{n}
	    & &\text{and} & \max_{P \in \cP} \xc(P) & \ge \frac{1}{15} \sqrt{n}.
    \end{align*}
\end{cor}
\begin{proof}
    As in \cite{FioriniRT12}, we consider polytopes with vertices on the parabola $\setcond{p_t}{t \in \R}$, where $p_t=(t,t^2)$. We introduce $ s \in \N$ with $s \ge n$, which will be fixed later. 
    For every $ I \subseteq [s] $ let $ P_I := \conv (\{ p_t : t \in I \})
    $.
    We consider the subfamily $ \cP_{s,n} := \{ P_I : I \subseteq [s], \, |I| = n \} $ of $\cP$. 
    Defining $ \Delta := \frac{1}{3s} $, we claim that $ \dist(P_I,P_J) \ge \Delta $ holds for all nonempty $ I,J \subseteq [s]
    $ with $ I \ne J $.
    To see this, we may assume that there is some $ t \in I \setminus J $.
	Observe that the line 
    through the points $ p_{t-1} $ and $ p_{t+1}$ separates $p_t$ and $P_J$. The distance of $p_t$ to this line is a lower bound on $\dist(P_I,P_J)$. Thus, we get
    \[
        \dist(P_I,P_J) \ge \frac{1}{\sqrt{4 t^2 + 1}}
        \ge \frac{1}{\sqrt{4s^2 + 1}} \ge \frac{1}{3s} = \Delta,
    \]
    as claimed.

    Furthermore, every member in $ \cP_{\ell,s} $ is contained in $ 2s^2 \cdotp \ball^2 $.
    Thus, setting $ \rho := 2s^2 $, we apply Theorem~\ref{thmMain} to the family $\cP_{s,n}$. This yields that if every $P \in \cP_{s,n}$ can be represented by $\ell$ semidefinite constraints of size $m$, then 
    \[
	    \ell^2 m^4 \ge \frac{\log |\cP_{s,n}|}{16 (1 + \log (12  s^3) + \log \log |\cP_{s,n}|) } \ge \frac{\log |\cP_{s,n}|}{16 (5 + 3 \log ( s) + \log \log |\cP_{s,n}|) }.
    \]
    Recall that $ \cP_{s,n} $ has $ \binom{s}{n} $ members, where 
    \[
        \left(\frac{s}{n}\right)^n \le \binom{s}{n} \le s^n.
    \]
    This yields $\log \log |\cP_{s,n}| \le \log n + \log \log s$ and $\log |\cP_{s,n}| \ge n (\log s - \log n)$. Thus, 
    \[
	    \ell^2 m^4 \ge \frac{n (\log s - \log n)}{16 (5 + 3 \log ( s) + \log n + \log \log s)) }.
    \]
    It is clear, that for sufficiently large $s$, the right hand side of the latter inequality is of order $\Theta(n)$. In fact, setting $s=n^2$, we obtain 
    \[
	    \ell^2 m^4 \ge \frac{n \log n}{16 (5 + 7 \log n + \log \log n)} \ge \frac{n }{16 \cdot 13 }
    \]
    The assertions for the semidefinite and the linear extension complexities follow by setting $\ell=1, \ m = \max \setcond{\sxc(P)}{P \in \cP_{s,n}} $ and $\ell =\max \setcond{\xc(P)}{P \in \cP_{s,n}}, \ m=1$, respectively.
\end{proof}

    \section*{Linear Extension Complexities of 0/1-Polytopes}
\label{secUpperBound}
We recall some well-known and simple facts about the extension complexity. For every finite set $X$, one has $\xc(\conv(X)) \le |X|$. If a polytope $P$ is represented as $P = \conv(P_1 \cup \dotsb \cup P_\ell)$ using finitely many polytopes $P_1,\ldots,P_\ell$, then $\xc(P) \le \ell + \xc(P_1) + \cdots + \xc(P_\ell)$; see \cite{Balas79}. If $P$ and $Q$ are polytopes, then $\xc(P \times Q) \le \xc(P) + \xc(Q)$.

The proof of the following theorem is inspired by a classroom proof of Shannon's upper bound on sizes of
boolean circuits~\cite[Thm.~6]{Shannon49}, see also~\cite[Sec.~2]{KramerL84}.

\begin{thm}
	\label{thmUpperBound01}
    For every $ d \in \N $ with $d \ge 4$ and every $P \in \cP(\{0,1\}^d)$, one has
    \[
        \xc(P) \le 9 \frac{2^d}{d}.
    \]
\end{thm}
\begin{proof}
    Let $V \subseteq \{0,1\}^d$ and $P=\conv(V)$. 
    We consider $s \in \{0,\ldots,d\}$, which will be fixed later.
    Points of $\{0,1\}^d$ can be represented as $(x,y)$ with $x \in \{0,1\}^{d-s}$ and $y \in \{0,1\}^s$.
    Using this representation, one can group points $(x,y) \in V$ into disjoint sets according to the choice of $x$.
    That is, $V$ is the disjoint union of the sets $\{x\} \times Y_x$ with $x \in \{0,1\}^{d-s}$, where $Y_x :=
    \setcond{y \in \{0,1\}^s}{(x,y) \in V}$.
    Note that one may have $Y_x=Y_{x'}$ for some $x \ne x'$.
    Let $\ell=2^{2^s}$ and let $Y_1,\ldots,Y_\ell$ be the sequence of all vertex sets of $0/1$-polytopes in $\R^s$.
    We now group the points $x \in \{0,1\}^{d-s}$  according to $Y_x$.
    That is, $\{0,1\}^{d-s}$ is the disjoint union of sets $X_1,\ldots,X_\ell$, where $X_i := \setcond{x \in
    \{0,1\}^{d-k}}{Y_x = Y_i}$ for $i \in [\ell]$.
    By construction, one has $V = \bigcup_{i=1}^\ell X_i \times Y_i$.
    Hence, $P = \conv(P_1 \cup \dotsb \cup P_\ell)$, where $P_i := \conv(X_i \times Y_i)$ for $i \in [\ell]$.
    This yields $\xc(P) \le \ell + \xc(P_1) + \cdots + \xc(P_\ell)$, where $\xc(P_i) \le \xc(\conv(X_i)) +
    \xc(\conv(Y_i)) \le |X_i| + |Y_i|$ holds for each $i \in [\ell]$.
    Summarizing, we get $\xc(P) \le \ell + \sum_{i=1}^\ell |X_i| + \sum_{i=1}^\ell |Y_i|$, where $\sum_{i=1}^\ell |X_i|
    \le 2^{d-s}$, since $X_1,\ldots,X_\ell$ are pairwise disjoint subsets of $\{0,1\}^{d-s}$, and $|Y_i| \le 2^s$.
    Consequently,
    \[
        \xc(P) \le  2^{d-s} + \ell (2^s + 1) = 2^{d-s} + 2^{2^s} (2^s+ 1) \le 2^{d-s} + 2^{2 \cdot 2^s}.
    \]
    Thus, setting $ s := \lfloor \log_2(d/4) \rfloor $ we obtain
    \begin{align*}
            \xc(P) & \le 2^{d- \log_2(d/4)+1} + 2^{\frac{d}{2}} = \frac{8}{d} 2^d + 2^{\frac{d}{2}}.
    \end{align*}
    For $d \ge 4$ one has $2^{\frac{d}{2}} \le \frac{1}{d} 2^d$ and so $\xc(P) \le \frac{9}{d} 2^d$.
\end{proof}

    \bibliographystyle{amsplain}
    \bibliography{references}
\end{document}